\documentclass[11pt,noindent]{article}

\usepackage{latexsym,amsmath,amsfonts,hyperref,amsthm,stmaryrd}

\usepackage{graphicx,color}

\usepackage{bbm}

\newtheorem{theorem}{Theorem}
\newtheorem{lemma}[theorem]{Lemma}

\newtheorem{remark}{Remark}

\newcommand{\thmref}[1]{Theorem~\ref{thm:#1}} 
\newcommand{\lemref}[1]{Lemma~\ref{lem:#1}} 

\newcommand{\R}{\mathbb{R}} 
\newcommand{\Q}{\mathbb{Q}} 
\newcommand{\Z}{\mathbb{Z}} 

\newcommand{\Lcal}{\mathcal{L}}

\newcommand{\Ncal}{\mathcal{N}}

\newcommand{\paren}[1]{\left( \left. #1 \right. \right)} 
\newcommand{\cro}[1]{\left[ \left. #1 \right. \right]} 
\newcommand{\set}[1]{\left\{ \left. #1 \right. \right\}}
\newcommand{\absj}[1]{\left\lvert #1 \right\rvert} 
\providecommand{\norm}[1]{\left \lVert #1 \right\rVert}

\newcommand{\Ind}[1]{{\bf 1}_{#1}}

\renewcommand{\P}{\mathbb{P}}

\newcommand{\E}{\mathbb{E}} 
\newcommand{\Ex}[1]{\E\cro{#1}}

\newcommand{\Var}{\mathbb{V}}
\newcommand{\var}[1]{\Var\cro{#1}} 

\newcommand{\cvp}{\stackrel{\P}{\longrightarrow}}

\newcommand{\1}[1]{\ensuremath{\mathbb{1}_{\left\{ #1 \right\}}}}	
\newcommand{\est}[2][n]{\widehat{#2}_{#1}}

\newcommand{\hyptag}[1]{\tag{\ensuremath{\mathbf{#1}}}} 

\DeclareMathOperator{\ud}{d}

\newtheorem{postita}{Post-it}

\newcommand{\seq}[2][n]{\left(#2_{#1}\right)_{#1\in\Z_+}}
\newcommand{\seqb}[2][n]{\left(#2^{#1}_{x}\right)_{0\leq x\leq #1}}
\newcommand{\F}{\mathcal{F}}

\newcommand{\Xab}[1][x]{X^{\alpha,\beta}_{n,#1}}

\begin{document}
\title{Non parametric estimation for random walks in random environment}
\author{Roland Diel \and Matthieu Lerasle}
\maketitle

\begin{abstract}
We consider a random walk in i.i.d. random environment with distribution $\nu$ on $\Z$. The problem we are interested in is to provide an estimator of the cumulative distribution function (c.d.f.) $F$ of $\nu$ from the observation of one trajectory of the random walk.
For that purpose we first estimate the moments of $\nu$, then combine these moment estimators to obtain a collection of estimators $(\est{F}^M)_{M\ge 1}$ of $F$, our final estimator is chosen among this collection by Lepskii's method. This estimator is therefore easily computable in practice.  We derive convergence rates for this estimator depending on the H\"older regularity of $F$ and on the divergence rate of the walk. 
Our rate is optimal when the chain realizes a trade-off between a fast exploration of the sites, allowing to get more informations and a larger number of visits of each sites, allowing a better recovery of the environment itself.
\end{abstract}

\noindent{\small {\bf Keywords and phrases :} random walk in random environment, non-parametric estimation, oracle inequalities, adaptive estimation . }


\section{Introduction}
Since its introduction by Chernov \cite{Chernov:1967} to model DNA replication, random walks in random environment (RWRE) on $\Z^d$ have been widely studied in the probabilistic literature. This model is now well understood in the case $d=1$, the case $d>1$ is more complex and only partial results have been obtained. 
A recent overview  can be found for example in \cite{Zei:2012}.

In this paper, we are interested in estimating the distribution $\nu$ from the observation of one trajectory of a random walk in random environment $\nu$ on $\Z$. 
%
The problem of estimation for RWRE was originally considered in \cite{Ade_Enr:2004} who introduced an estimator of the moments of the distribution. 
The state space of the walk in \cite{Ade_Enr:2004} is more general than $\Z$ but their estimators have a huge variance, they are therefore unstable and cannot really be used in practice. More recently, \cite{Fal_Lou_Mat:2014, Fal_Glo_Lou:2014, CoFaLoLoMa2014,ComFalLouLou:2016} considered the random walk on $\Z$ and investigated 
the problem in a parametric framework. They proved consistency of the maximum likelihood estimator in various regimes of the walk and even its asymptotic normality and efficiency in the ballistic regime, see Section \ref{sec:Setting} for details. 

Although very interesting, this approach suffers several drawbacks both for practical applications and from a statistical perspective. First, the results are stated in a purely asymptotic framework where the number $n$ of sites visited by the walk tends to infinity. Next, the quality of the estimator strongly relies on the assumption that the unknown distribution lies in a parametric model. Both assumptions impose severe restrictions for applications. The robustness of the procedure to a misspecified model for the unknown distribution, or the dependence of the performances of the maximum likelihood estimator with respect to an increasing number of parameters to recover are not considered.  Moreover, the maximum likelihood estimator can be evaluated only after solving a maximization problem that is computationally intractable in general. Finally, the estimators of \cite{Fal_Lou_Mat:2014, Fal_Glo_Lou:2014, CoFaLoLoMa2014, ComFalLouLou:2016} are not exactly the same depending on the regime of the walk (recurrent or transient). This is an important problem from a statistical perspective since the regime depends on the unknown distribution of the observations, see Section~\ref{sec:Setting} for details. 

In this paper, we propose by contrast a non-asymptotic and non-parametric approach to tackle the estimation of the unknown cumulative distribution function (c.d.f.) of the environment from one observation of the walk. All  our concentration results are valid in any regime, the only difference between the regimes lies in the convergence rate of the c.d.f. estimator. Our approach is based on the estimation of the moments of the unknown distribution, these estimations can always be performed in linear time. Those primary estimators are then combined to build a collection of estimators with non-increasing bias and non-decreasing variance and the final estimator is chosen among them according to the Lepskii method \cite{Lepski:1991}. The resulting estimator is therefore very fast to compute and provides at least a starting point to an optimization algorithm computing the maximum likelihood. It satisfies an oracle type inequality, meaning that it performs as well as the best estimator of the original collection. The oracle type inequality is used to obtain rates of convergence under regularity assumptions on the unknown c.d.f.. More precisely, the rate of convergence of our estimator, stated in \thmref{cdfopt} below in terms of the number $n$ of visited sites, is given in the recurrent case by $\frac{\log n}{\sqrt{n}}$ and in the transient case by $\paren{\frac{\log n}{n}}^{\gamma/(2\gamma+4\kappa)}$, where $\gamma$ is the H\"older regularity of the unknown c.d.f. and $\kappa> 0$ is a parameter related to the rate at which the chain derives to infinity, see Section~\ref{sec:Setting} for details. This rate can be compared with the one we would achieve if we observed the environment $(\omega_x)_{1\le x\le n}$. Actually, the empirical c.d.f. is known to converge at rate $1/\sqrt{n}$, without assumptions on the regularity of $F$ by the Kolmogorov-Smirnov theorem and Dvoretzky-Kiefer-Wolfowitz inequality \cite{Dvo_Kie_Wol:1956, Mas:1990} gives a precise non-asymptotic concentration inequality. Our result is therefore much weaker, which is not surprising since we only observe a trajectory of the RWRE, but it can be noticed that, in the recurrent case, we recover, up to the logarithmic factor, this usual rate of convergence. 
Indeed, in this regime, the walk visits every site infinitely often and it allows to learn the environment itself. One could also recover this rate in the limit $\gamma\to\infty$. In \thmref{cdfopt}, $\gamma$ is assumed to be smaller than $2$, the extension of our results to $\gamma>2$ would require further technical analysis that is not performed here. The optimality of the dependence in $\kappa$ in general remains also an open question.

The performance of the estimator seems to deteriorate as $\kappa$ increases, that is when the chains derives faster to infinity, which is confirmed by our short simulation study in Section~\ref{Sec:proofs}. However, when expressed in terms of the number of observations, that is the number $T_n$ of steps of the walk, the rates become $\frac{\log \log T_n}{\log T_n}$ in the recurrent case, $\paren{\frac{\log T_n}{T_n}}^{\frac{\gamma\kappa}{2\gamma+4\kappa}}$ when $0<\kappa<1$, $\paren{\frac{(\log T_n)^2}{T_n}}^{\frac{\gamma}{2\gamma+4}}$ when $\kappa=1$ and $\paren{\frac{\log T_n}{T_n}}^{\frac{\gamma}{2\gamma+4\kappa}}$ when $\kappa>1$, see the remark after \thmref{cdfopt}. It follows that the best rate is actually achieved when $\kappa=1$. This looks surprising compared to the results of \cite{Fal_Lou_Mat:2014} where the rate $1/\sqrt{T_n}$ can be recovered in the ballistic regime ($\kappa>1$). The non-parametric problem that we consider seems therefore more complex than the parametric case. Actually, our rate of convergence is optimal in a regime where the walk realizes a trade-off between visiting more sites to obtain more information and spending more time on each site to learn the environment itself. 
 
More generally, from a statistical perspective, we believe that the RWRE can be seen as a toy model for non-linear inverse problems in statistics. While linear inverse problems have been deeply studied in the last decades, see for example \cite{Cav:2011, Aster:2013} for recent overviews, much less is known when the observation is not a noisy version of some linear transformation of the signal of interest. The problem considered in this paper is a typical example which has been intensively studied from a probabilistic point of view. As such, many tools for statistical analysis, such as concentration inequalities, are already proved, or can be easily derived from existing results on the walk. 

The paper is organized as follows. In Section~\ref{sec:Setting}, we present the model, recall a few basic results on the RWRE and state our main theorem. In Section~\ref{sec:moments}, we present the construction of the estimators of the moments of the distribution of the environment, which are the building blocks of our procedure. We also present their concentration properties and the key martingale arguments leading to these results. Section \ref{sec:cdf} presents the construction of the collection of estimators for the c.d.f. and states the concentration properties of these estimators. It also presents the oracle type inequality satisfied by the final estimator chosen among the precedent collection by the Lepskii method. 
A short simulation study is presented in Section \ref{sec:simu} showing the actual performances of our estimators. The most technical proofs are postponed to Section~\ref{Sec:proofs}.

\section{Setting}\label{sec:Setting}

Let $\omega=\left(\omega_x\right)_{x\in\Z}$ be an independent and identically distributed (i.i.d.) sequence of random variables taking values in $(0,1)$, with common distribution $\nu$. The random variable $\omega$ is called the \emph{environment} and its distribution on $(0,1)^\Z$ is denoted by $\Q^\nu=\nu^{\otimes\Z}$. Given a realization of the environment $\omega$, let $S=\left(S_t\right)_{t\in\Z_+}$ denote the random walk in the environment $\omega$, that is the Markov chain on $\Z$ starting at $S_0=0$ and with probability transitions defined as follows:
$$
\P_\omega\left(S_{t+1}=y|S_{t}=x\right)=\begin{cases}
                                          	\omega_x&\text{if } y=x+1\\
                                          	1-\omega_x&\text{if } y=x-1\\
                                          	0&\text{otherwise}
                                          \end{cases}.
$$
The probability measure $\P_\omega$ of the chain, conditionally on the environment $\omega$, is usually called the quenched distribution, while the unconditional distribution given by 
$$
\P^\nu\left(\cdot\right)=\int\P_\omega\paren{\cdot}\Q^\nu(\ud \omega)
$$
is called the annealed distribution. The asymptotic behavior of the walk $\seq[t]{S}$ depends on  the random variables $\rho_x=\frac{1-\omega_x}{\omega_x}$. More precisely, if $\E^\nu\cro{|\log\rho_0|}$ is finite, Solomon \cite{Sol:1975} proved the following classification:
\begin{enumerate}
	\item if $\E^\nu\cro{\log\rho_0}<0$, $\lim_{t\to\infty}S_t=\infty$,
	\item if $\E^\nu\cro{\log\rho_0}=0$,  $\limsup_{t\to\infty}S_t=\infty$ and $\liminf_{t\to\infty}S_t=\infty$.
\end{enumerate}

The exact divergence rate of $\seq[t]{S}$ in the first case was obtained by Kesten, Kozlov and Spitzer \cite{KesKozSpi:1975}. Suppose that the distribution of $\log \rho_0$ is non arithmetic (that is the group generated by the support of $\log\rho_0$ is dense in $\R$) and that there exists some $\kappa\in(0,\infty)$ such that 
\begin{align}\label{eq:kappa}
\E^\nu\cro{\rho_0^\kappa}=1\  \text{ and }\ \E^\nu\cro{\rho_0^\kappa\log^+(\rho_0)}<\infty 
\end{align}
where $\log^+(x)=\log(x\vee 1)$. 

When $\kappa$ exists, a simple convexity argument shows that it is unique. This value determines the asymptotic divergence rate of $\seq[t]{S}$. More precisely, let $T_n$ denote the first hitting time of $n\in \Z_+$,  $T_n=\inf\set{t\in\Z_+,\ S_t=n}$: 
\begin{enumerate}
	\item if $\kappa<1$, 
	$T_n/n^{1/\kappa}$ and $S_t/t^\kappa$ converge in distribution to some non trivial distribution, 
	\item if $\kappa=1$, 
	$\frac{T_n}{n\log n}$ and $\frac{(\log t)^2}{t}S_t$ converge in probability,
	\item if $\kappa>1$,
	$\frac{T_n}{n}$ and $\frac{S_t}{t}$ converge in probability.
\end{enumerate}

The first two cases are called the sub-ballistic cases and the last one the ballistic case, where $T_n$ and $S_t$ grow linearly. 

In the recurrent case, the order of magnitude of the fluctuations of $S_t$ was obtained by Sinai \cite{Sinai:1982}. Suppose that $\E^\nu\cro{\log\rho_0}=0$, $\E^\nu\cro{\paren{\log\rho_0}^2}>0$ and that the support of the law of $\rho_0$ is included in $(0,1)$, then $S_t/(\log t)^2$ converges in distribution to a non trivial limit.

Our main result is valid either under the assumptions of \cite{KesKozSpi:1975} or under a slightly weaker version of the  ones presented in \cite{Sinai:1982}: let us introduce the following assumption
\begin{align}\label{hypderdroite}
\begin{cases}
	&\quad\E^\nu\cro{\log\rho_0}=0,\ \E^\nu\cro{\paren{\log\rho_0}^2}>0\\
	&\quad\text{and}\ \exists a>0,\ \E^\nu\cro{\rho_0^a}+\E^\nu\cro{\rho_0^{-a}}<+\infty\\
	&\text{or }\\
	&\quad\E^\nu\cro{\log\rho_0}<0\text{, the distribution of $\log \rho_0$ is non arithmetic}\\
&\quad\text{and }\exists \kappa\in(0,\infty),\ \E^\nu\cro{\rho_0^\kappa}=1\text{ and }\ \E^\nu\cro{\rho_0^\kappa\log^+(\rho_0)}<\infty\enspace .
\end{cases}
\hyptag{H}
\end{align}
Under \eqref{hypderdroite}, $\seq[t]{S}$ is either transient to the right, when $\E^\nu\cro{\log\rho_0}<0$ or recurrent, when $\E^\nu\cro{\log\rho_0}=0$. In both cases, $T_n$ is almost surely finite for any $n\in\Z_+$.

Our problem here is to estimate the c.d.f. $F$ of the distribution $\nu$ using the path $S_{[0,T_n]}=\set{S_t,\ 0\leq t\leq T_n}$. 
As we need to assume that $F$ is H\"older continuous to bound the bias of our estimators, we recall the definition of $\gamma$-H\"older seminorms and spaces:
	for any $\gamma\in(0,1]$, the H\"older space $\mathcal{C}^\gamma$ is the set of continuous functions $f:[0,1]\to\R$ such that 
	$$\|f\|_\gamma=\sup_{u\neq v}\frac{|f(v)-f(u)|}{|v-u|^\gamma}<\infty$$ 
	and for $\gamma\in(1,2]$ the H\"older space $\mathcal{C}^\gamma$ is the set of continuously differentiable functions $f:[0,1]\to\R$ such that 
	$$\|f\|_\gamma=\|f'\|_\infty+\sup_{u\neq v}\frac{|f'(v)-f'(u)|}{|v-u|^{\gamma-1}}<\infty\enspace.$$

The following theorem is the main result of the paper.
\begin{theorem}\label{thm:cdfopt}
	Suppose that the c.d.f. $F(t)=\int_0^t\nu(\ud u)$ is $\gamma$-H\"older for some $\gamma\in(0,2]$ and that $\nu$ satisfies Assumption \eqref{hypderdroite}. There exists a constant $C_\nu$ depending only on the distribution $\nu$ such that, for any integer $n\geq 2$, there exists an estimator $\est{F}=f_n\paren{S_{[0,T_n]}}$ satisfying
	\[\E^\nu\cro{\|\est{F}-F\|_\infty}\leq 
\begin{cases}
 C_\nu\paren{\frac{\log n}{n}}^{\frac{\gamma}{2\gamma+4\kappa}}&\text{if}\quad \E^\nu\cro{\log\rho_0}<0\\
 C_{\nu}\frac{\log n}{\sqrt{n}}&\text{if}\quad \E^\nu\cro{\log\rho_0}=0
\end{cases}
\enspace.\]
Moreover, for any integer $n\geq 1$ and any real $z>0$, there exists an estimator $\est{F}^z=f^z_n\paren{S_{[0,T_n]}}$ such that if $\E^\nu\cro{\log\rho_0}<0$,
\[\P^\nu\paren{\|\est{F}^z-F\|_\infty\geq C_\nu\paren{\frac{z+\log n}{n}}^{\frac{\gamma}{2\gamma+4\kappa}}}\leq e^{-z}\enspace.\]
\end{theorem}

\begin{remark}
In the recurrent case, the rate  $\log n/\sqrt{n}$ is, up to the logarithmic factor, the rate of convergence of the empirical c.d.f. when the environment $(\omega_x)_{0\le x\le n}$ is observed. This is the best rate reached by our estimator expressed in terms of the number $n$ of visited sites. This is not surprising since the walk visits each site many times and can basically learn the environment itself. When $\kappa>0$, the rate deteriorates as $\kappa$ increases, which was also expected since the walk derives faster to infinity in this case. However, the rate of convergence can also be expressed in function of the number of observations, that is the time $T_n$ it took to reach site $n$. In the recurrent regime, $\log T_n\sim n^{1/2}$, so the rate becomes 
\[\frac{\log \log T_n}{\log T_n}\enspace.\]
In the transient regime, for $\kappa<1$, $T_n\sim n^{1/\kappa}$, so the rate of convergence becomes
 \[\paren{\frac{\log T_n}{T_n}}^{\frac{\gamma\kappa}{2\gamma+4\kappa}}\enspace.\]
 When $\kappa=1$, $T_n\sim n\log n$ and we get the rate of convergence
  \[\paren{\frac{(\log T_n)^2}{T_n}}^{\frac{\gamma}{2\gamma+4}}\enspace.\]
While for $\kappa>1$, $T_n\sim n$ and the rate of convergence is
  \[\paren{\frac{\log T_n}{T_n}}^{\frac{\gamma}{2\gamma+4\kappa}}\enspace.\]
Therefore, the best rate expressed in the number $T_n$ of observed steps of the walk is obtained in the sub-ballistic regime where $\kappa=1$. The situation is slightly more complicated in the non-parametric problem considered in this paper than in the parametric setting of \cite{Fal_Lou_Mat:2014, Fal_Glo_Lou:2014, CoFaLoLoMa2014, ComFalLouLou:2016}, where the optimal rates were obtained in the ballistic regime. In the non parametric case, there seems to be a trade-off between exploring more sites (increasing $\kappa$) to get information about more realizations of $\nu$ and spend more time on theses sites (decreasing $\kappa$) to have a better knowledge of these realizations.
\end{remark}



\section{Estimation of the moments of the environment}\label{sec:moments}
This section presents the estimators of the moments of the environment and the key martingale arguments underlying their concentration properties. These will be the basic tools to build and control the estimator of the c.d.f. in the following sections.

\medskip

Following \cite{CoFaLoLoMa2014}, we write the likelihood of the observation using the following processes. Let 
\begin{align*}
	L(t_0,x)&=\sum_{0\leq t\leq t_0-1}\1{S_t=x,S_{t+1}=x-1},\\
	R(t_0,x)&=\sum_{0\leq t\leq t_0-1}\1{S_t=x,S_{t+1}=x+1}
\end{align*}
denote the number of left (resp. right) steps for $S$ until time $t_0$ and from site $x$. 
The likelihood $\Lcal_{\nu}\paren{S_{[0,T_n]}}$ of the observation can be expressed in the following way, see \cite{CoFaLoLoMa2014},
\begin{multline*}
\int\paren{\prod_{x\in\Z}\omega_x^{R(T_n,x)}(1-\omega_x)^{L(T_n,x)}}\Q^\nu(d\omega)\\
=\prod_{x\in\Z}\int_0^1a^{R(T_n,x)}(1-a)^{L(T_n,x)}\nu(da)\enspace. 
\end{multline*}
Now, our choice of $T_n$ implies that $L(T_n,n)=0$ and
\[L(T_n,x+1)=
\begin{cases}
 R(T_n,x),\ &\forall x<0\\
 R(T_n,x)-1,\  &\forall x\in [0,n-1]
\end{cases}
\enspace.\]
Hence,
\[\Lcal_{\nu}\paren{S_{[0,T_n]}}=\prod_{x\le n-1}\int_0^1a^{L(T_n,x+1)+\1{x\ge 0}}(1-a)^{L(T_n,x)}\nu(da)\enspace.\]
The collection $(L(T_n,x))_{x\le n}$ is therefore an exhaustive statistic in our problem on which we will base our estimation strategy. An important result of \cite{KesKozSpi:1975} is that the process 
$$(Z_x^n)_{0\leq x\leq n}=(L(T_n,n-x))_{0\leq x\le n}$$ is a branching process in random environment with immigration. This representation was successfully used in \cite{CoFaLoLoMa2014} to deal with the parametric case. In particular (see \cite[Proposition 4.3]{CoFaLoLoMa2014}), under the annealed law $\P^\nu$, $(Z_x^n)_{0\leq x\leq n}$ has the same distribution as $(Z_x)_{0\leq x\leq n}$ where $(Z_x)_{x\in\Z_+}$ is an homogeneous Markovian process with transition kernel
\begin{align}
K^\nu(i,j)=\binom{i+j}{j}\int_0^1a^{i+1}(1-a)^{j}\nu(\ud a)\label{eq:ker1}\enspace.
\end{align}
Moreover, if $\E^\nu\cro{\log\rho_0}<0$, $(Z_x)_{x\in\Z_+}$ is positive, recurrent and aperiodic and admits the unique invariant probability measure $\pi$ defined, for any $i\in\Z_+$, by
\begin{align}\label{eq:loiinv}
\pi(i)=\E^\nu\cro{W^{-1}(1-W^{-1})^i} \quad \text{where}\quad W=\sum_{x=0}^\infty e^{V_x-V_0}
\end{align}
and for any $x\in\Z_+$, $V_x=\sum_{z=0}^x\log \rho_z$ (see \cite[Theorem 4.5]{CoFaLoLoMa2014}).

Equation \eqref{eq:ker1} shows that it is natural to estimate the moments $$m^{\alpha,\beta}=\E^\nu\cro{\omega_0^\alpha(1-\omega_0)^\beta}=\int_0^1a^{\alpha}(1-a)^\beta\nu(\ud a),\ \alpha,\beta\in\Z_+\enspace.$$ Our estimation strategy is based on the remark that, for any $\alpha,\beta\in\Z_+$,
\begin{align}\label{eq:combfac1}
\forall i\geq \alpha,\qquad \sum_{j\ge \beta}\binom{i-\alpha+j-\beta}{i-\alpha}a^{i+1}(1-a)^{j}&=a^{\alpha}(1-a)^\beta\enspace.
\end{align}

Integrating this equality with respect to $a$ leads to the relation.
\[\forall i\geq \alpha,\qquad\sum_{j\ge 0}\binom{i-\alpha+j-\beta}{i-\alpha}\frac{K^\nu(i,j)}{\binom{i+j}{j}}=m^{\alpha,\beta}\enspace.\] 
In other words, for any $\alpha,\beta\in\Z_+$,
\begin{align}\label{eq:Estmalpha}
\forall i\geq\alpha,\ \forall x\in[1,n]&\qquad m^{\alpha,\beta}\Ind{i\ge \alpha}=\Ex{\Phi_{\alpha,\beta}(Z_{x-1}^n,Z_{x}^n)|Z_{x-1}^n=i}\enspace,
\end{align}
where, for any integers $i$ and $j$,
\begin{align*}
	\Phi_{\alpha,\beta}(i,j)=\1{i\geq\alpha,j\geq\beta}\frac{\binom{i+j-(\alpha+\beta)}{i-\alpha}}{\binom{i+j}{i}}=\1{i\geq\alpha,j\geq\beta}\frac{\prod_{l=0}^{\alpha-1}(i-l)\prod_{l=0}^{\beta-1}(j-l)}{\prod_{l=0}^{\alpha+\beta-1}(i+j-l)}\enspace.
\end{align*}
It is therefore natural to estimate $m^{\alpha,\beta}$ by the following estimator.
\begin{align}\label{eq:defN}
\est{m}^{\alpha,\beta}=\frac{1}{N_{n}^{\alpha}}\sum_{x=1}^{n}\Phi_{\alpha,\beta}(Z^n_{x-1},Z^n_{x})\quad\text{ where }\quad N_{n}^{\alpha}=\sum_{x=0}^{n-1}\1{Z_x^n\ge \alpha} \enspace,
\end{align} 
with the convention that $0/0=0$.
The following lemma summarizes the preceding remarks.

\begin{lemma}\label{lem:Martingale} For any integer $n\geq0$, denote by $\paren{\mathcal{F}_{n,x}}_{0\leq x\leq n}$ the filtration generated by the sequence $\paren{Z^n_x}_{0\leq x\leq n}$. For any $\alpha,\beta\in\Z_+$, the triangular arrays $(X^{\alpha,\beta}_{n,x})_{0\leq x\leq n<\infty}$ defined, for all integer $n\geq0$, by
\begin{align*}
X^{\alpha,\beta}_{n,0}&=0\enspace,\\
\forall x\in\set{1,\dots,n},\quad X^{\alpha,\beta}_{n,x}&= \Phi_{\alpha,\beta}(Z^n_{x-1},Z^n_{x})- m^{\alpha,\beta}\1{Z^n_{x-1}\geq \alpha}
\end{align*}
	are martingale difference arrays with respect to $\paren{\mathcal{F}_{n,x}}_{0\leq x\leq n<\infty}$.
\end{lemma}

\begin{proof}
For any integers $n,\alpha,\beta\geq0$, the process $(X^{\alpha,\beta}_{n,x})_{0\leq x\leq n}$ is adapted to $\paren{\mathcal{F}_{n,x}}_{0\leq 
x\leq n}$. Moreover, \eqref{eq:Estmalpha} yields directly $\Ex{X^{\alpha,\beta}_{n,x+1}|\mathcal{F}_{n,x}}=0$. 
\end{proof}

\lemref{Martingale} shows that one can use martingales theory to control the risk of our moment estimators. As an example, we give the risk bounds derived from Mc Diarmid's inequality \cite{McD89} in \thmref{BorneRisque} and \thmref{CLTMoment} gives the central limit theorem satisfied by these estimators.

\begin{theorem}\label{thm:BorneRisque}
	Assume that the random walk is either recurrent or transient to the right, that is $\E^\nu[\log\rho_0]\leq0$, and let $\alpha,\beta\in\Z_+$. For any integer $n\geq1$ and any real number $z>0$,
	$$
	\P^\nu\paren{\absj{\est{m}^{\alpha,\beta}-m^{\alpha,\beta}}\geq \frac{n}{N_{n}^{\alpha}}\binom{\alpha+\beta}{\alpha}^{-1}\sqrt{\frac{z}{2n}}}\leq 2e^{-z}\enspace.
	$$
\end{theorem}

\begin{remark}
 Note that, if $\Ex{\log \rho_0}<0$, $\frac{N_n^{\alpha}}n=\frac1n\sum_{k=1}^n\1{Z_k^n\ge \alpha}$ converges according to \eqref{eq:loiinv}  to $\Ex{(1-W^{-1})^{\alpha}}>0$ and  if $\Ex{\log \rho_0}=0$, $N_n^{\alpha}/n\ge 1/2$ with large probability (see \lemref{AsyNnMrec}). Therefore, for any $\alpha,\beta\geq0$, $\est{m}^{\alpha,\beta}$ converges at parametric rate $\sqrt{n}$. 
 \end{remark}

 \begin{remark}
 For $\alpha=0$, $N_{n}^{0}=n$, therefore the convergence rate of the estimator of the moments $\Ex{(1-\omega_0)^\beta}$ for $\beta>0$ is deterministic in this case.
\end{remark}

\begin{proof}
Notice that
 $$
 N_n^\alpha(\est{m}^{\alpha,\beta}-m^{\alpha,\beta})=\sum_{x=1}^n\Xab=\sum_{x=1}^n\Phi_{\alpha,\beta}(Z^n_{x-1},Z^n_{x})- m^{\alpha,\beta}\1{Z^n_{x-1}\geq \alpha}\enspace.
 $$
 Moreover, an elementary combinatoric argument shows that, for $i\geq\alpha$ and $j\geq\beta$, $\binom{i+j-\alpha-\beta}{i-\alpha}\binom{\alpha+\beta}{\alpha}\leq\binom{i+j}{i}$. Thus, for any $n\ge 1$ and $x\le n-1$, 
$$0\le \Phi_{\alpha,\beta}(Z^n_x,Z^n_{x+1})\le \frac{1}{\binom{\alpha+\beta}{\alpha}}=\Phi_{\alpha,\beta}(\alpha,\beta)\enspace.$$
 \thmref{BorneRisque} follows now from \lemref{Martingale} and Mc Diarmid's inequality (see Theorem 6.7 in \cite{McD89}).
\end{proof}

%

The asymptotic behavior of the estimators in the transient case is given more precisely in the following theorem.

\begin{theorem}\label{thm:CLTMoment} Suppose that $\E^\nu[\log\rho_0]<0$. For any $\alpha,\beta\in\Z_+$, the estimator of $m^{\alpha,\beta}$ is asymptotically normal, more precisely 
\[\sqrt{n}(\est{m}^{\alpha,\beta}-m^{\alpha,\beta})\xrightarrow[n\to\infty]{\Lcal}\Ncal(0,V_{\alpha,\beta}^2)\]
where 
\[V_{\alpha,\beta}^2=\frac{\E^\nu\cro{\left(\Phi_{\alpha,\beta}(\widetilde{Z}_0,\widetilde{Z}_1)\right)^2}}{\E^\nu\cro{(1-W^{-1})^\alpha}^2}-\left(m^{\alpha,\beta}\right)^2\]
where $\left(\widetilde{Z}_x\right)_{x\geq0}$ is a Markov chain with transition kernel $K^\nu$ started with the invariant distribution $\pi$ (see \eqref{eq:ker1} and \eqref{eq:loiinv}).
\end{theorem}
\thmref{CLTMoment} is proved in Section~\ref{Sec:ProofCLTMoment}.

\begin{remark}
When $\alpha=0$, 
\[m^{0,2\beta}-\left(m^{0,\beta}\right)^2\leq V_{0,\beta}^2\leq m^{0,\beta}-\left(m^{0,\beta}\right)^2\enspace.\]
Moreover, these bounds are tight, the lower bound is reached when $\kappa\to0$ and the upper bound when $\kappa\to\infty$.
Our estimators can therefore be compared to the empirical means, knowing the environment: $\frac1{n+1}\sum_{x=0}^n(1-\omega_x)^\beta$. The central limit theorem shows that this random variable is asymptotically normal with limit variance $m^{0,2\beta}-\left(m^{0,\beta}\right)^2$. Therefore, the performance of our estimators matches those of this ideal case when the chain is almost recurrent but there is a loss in the constants otherwise.
\end{remark}

\section{Estimation of the cumulative distribution function}\label{sec:cdf}

We now want to use the estimation of the moments $m^{\alpha,\beta}$ to approximate the cumulative distribution function $F$ of $\nu$.  Define for any $u\in[0,1]$, 
\begin{equation}
F^M(u)=\sum_{k=0}^{[(M+1)u]-1}\binom{M}{k}m^{k,M-k}                                                                                                                                                                                                                                                                                                                                                                                                                                                                                                                                                                               \end{equation}
with the usual convention that $\sum_{k=0}^{-1}=0$.
\lemref{biasF} shows that, if $F$ is H\"older continuous, $F^M$ converges uniformly to $F$ when $M$ tends to infinity. Thus, we only have to estimate $F^M$. We propose the following moment estimator
\begin{equation}\label{eq:estF}
\est{F}^M(u)=\frac{1}{N_{n}^{M}}\sum_{x=1}^n\psi^{[(M+1)u]}_M(Z^n_{x-1},Z^n_{x})
\end{equation}
where 
\begin{equation}\label{def:psi}
	\psi^{l}_M(i,j)=\frac{\1{i\geq M}}{\binom{i+j}{M}}\sum_{k=0}^{l-1}\binom{i}{k}\binom{j}{M-k}
\end{equation}
and $N_{n}^{M}=\sum_{x=0}^{n-1}\1{Z_x^n\geq M}$ as in \eqref{eq:defN}, still using the convention $0/0=0$. For any $i\geq M$,
$$
\psi^{l}_M(i,j)=\sum_{k=0}^{l-1}\binom{M}{k}\Phi_{k,M-k}(i,j)
$$
where $\Phi_{k,M-k}$ is defined in \eqref{eq:Estmalpha}. Therefore, the estimator \eqref{eq:estF} is essentially the estimator of $F^M$ obtained from the moment estimators of Section \ref{sec:moments}, but using only the sites $x$ satisfying $Z^n_x\geq M$. $\est{F}^M$ is an unbiased estimator of $F^M$ as shown by \lemref{MartingaleF}. Moreover, as $\sum_{k=0}^{M}\binom{i}{k}\binom{j}{M-k}=\binom{i+j}{M}$ for any $i,j\geq0$, any $\est{F}^M$ is a (random) c.d.f. 

The following lemma gives an upper bound on the risk of each estimator $(\est{F}^M)_{M\in \Z_+}$.
\begin{lemma}\label{lem:ControlEstFModel}Assume that the random walk is either recurrent or transient to the right, that is $\E^\nu[\log\rho_0]\leq0$, and that the function $F$ is in $\mathcal{C}^\gamma$ for some $\gamma\in(0,2]$.
	For any integers $M,n\geq 1$, and any real $z>0$, we have
	$$
	\P^\nu\cro{\|\est{F}^M-F\|_\infty\geq \frac{n}{N_{n}^{M}}\sqrt{\frac{z+\log M}{2n}}+\frac{2\|F\|_\gamma}{(M+1)^{\gamma/2}}}\leq 2e^{-z}\enspace.
	$$
\end{lemma}

\lemref{ControlEstFModel} is proved in Section \ref{ProofControlEstFModel}.  The first term in the bound is random, it is derived from the martingale argument presented in the previous section. The second term is the upper bound on the bias of this estimator derived from regularity assumptions on $F$. It is interesting to notice that, although $F^M$ is a histogram, one can take advantage of the regularity of $F$ up to $\gamma=2$.

\medskip

The estimator $\est{F}$ given in \thmref{cdfopt} is obtained via Lepskii's method, see \cite{Lepski:1991}, using the collection $(\est{F}^M)_{M\ge 1}$. This method selects, for any fixed $z>0$, a regularizing parameter $\est{M}^z\ge 0$ without the knowledge of $\gamma$, such  that the estimator $\est{F}^{\est{M}^z}$  optimizes, up to a multiplicative constant, the bound given by \lemref{ControlEstFModel}.
\begin{lemma}\label{lem:Lepski}
Assume that the random walk is either recurrent or transient to the right, that is $\E^\nu[\log\rho_0]\leq0$, and that the function $F$ is in $\mathcal{C}^\gamma$ for some $\gamma\in(0,2]$. For any integer $n\ge 1$ and any real $z>0$, there exists a r.v. $\est{M}^z=g_{n,z}(S_{[0,T_n]})$ such that,
\[\P^\nu\cro{ \norm{\est{F}^{\est{M}^z}-F}_{\infty}>\inf_{M\ge 1}\set{\frac{4n}{N_{n}^{M}}\sqrt{\frac{z+3\log M}{2n}}+ \frac{6\|F\|_\gamma}{(M+1)^{\gamma/2}}}}\le \frac{\pi^2}{3}e^{-z}\enspace.\]
\end{lemma}
The details of the construction of $\est{M}^z$ and the proof of \lemref{Lepski} are given in Section~\ref{ConstructionofestF}.

To derive, from \lemref{Lepski}, the rate of convergence of the estimator $\est{F}^{z}=\est{F}^{\est{M}^{z}}$ and conclude the proof of \thmref{cdfopt}, we have to study the asymptotic behavior of $\frac{N^M_n}{n}$. This study is performed in section~\ref{Sec:AsNnM} separately for the recurrent case and the transient case. The reason is that, in the transient case, the Markov chain $\seq[x]{Z}$ admits an invariant probability $\pi$ while it doesn't in the recurrent case. The proof of \thmref{cdfopt} is completed in Section \ref{ConclusionProof}.


\section{Simulation Study}\label{sec:simu}

This section illustrates the results of \thmref{cdfopt} with some experiments on synthetic data. 


We first consider the case of Beta distribution $B(a,b)$. In this example, $F$ is clearly infinitely differentiable and simple computations show that the coefficient $\kappa$ is equal to $a-b$. 

Figures \ref{simu1}-\ref{simu4} show the estimates of the c.d.f. for various values of $\kappa$ and $n=500$, illustrating the improvement of the convergence rates as $\kappa$ decreases. They also provide the value of the selected model $\est{M}$  and the value of the loss $N_\infty=\|F-\est{F}^{\est{M}}\|_\infty$. The red curve is the empirical c.d.f. knowing the environment $\paren{\omega_x}_{0\leq x\leq n-1}$. 

\begin{figure}[h!]
	
	\begin{minipage}[c]{.46\linewidth}
      \includegraphics[scale=0.35]{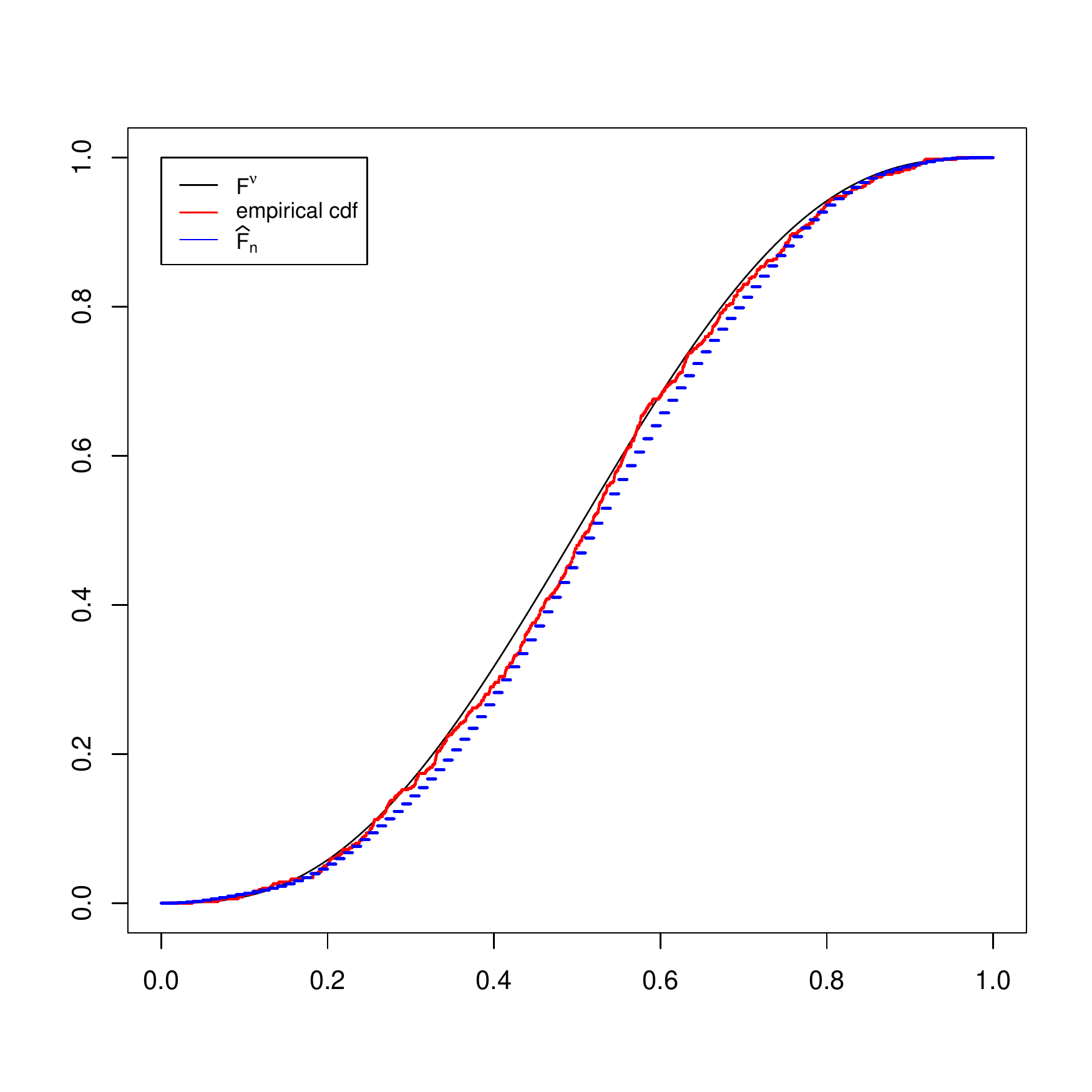}
      
      \vspace{-0.75cm}
      
      \caption{Recurrent case $B(3,3)$: $n=500$, $T_n\approx 10^9$, $\est{M}=39$, $N_\infty\approx 0.065$.}\label{simu1}
      
      \vspace{0.25cm}
      
      \includegraphics[scale=0.35]{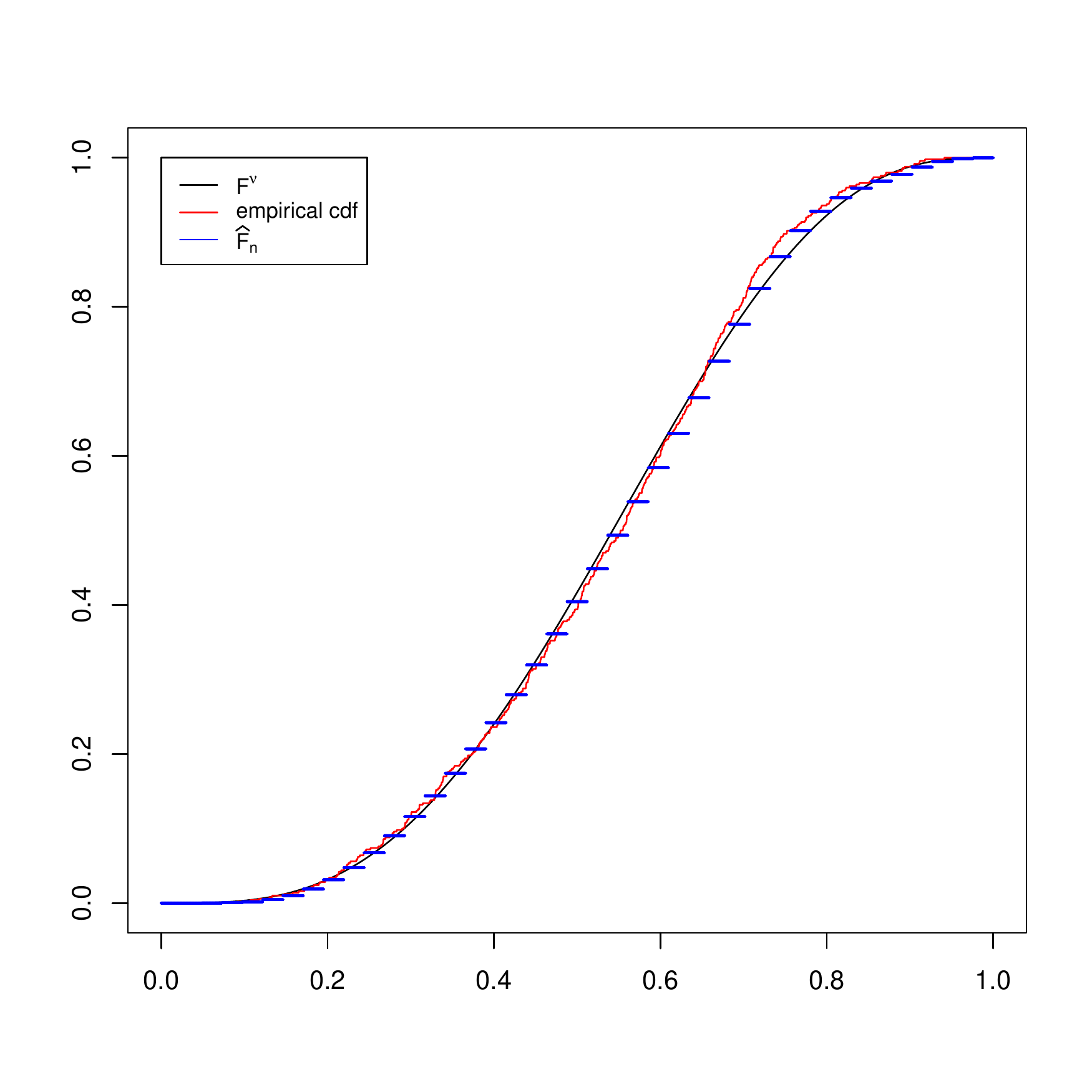}
      \caption{$\kappa=0.5$, $B(3.5,3)$: $n=500$, $T_n\approx 195046$, $\est{M}=37$, $N_\infty\approx 0.028$.}\label{simu2}
   \end{minipage} \hfill
   \begin{minipage}[c]{.46\linewidth}
      \includegraphics[scale=0.35]{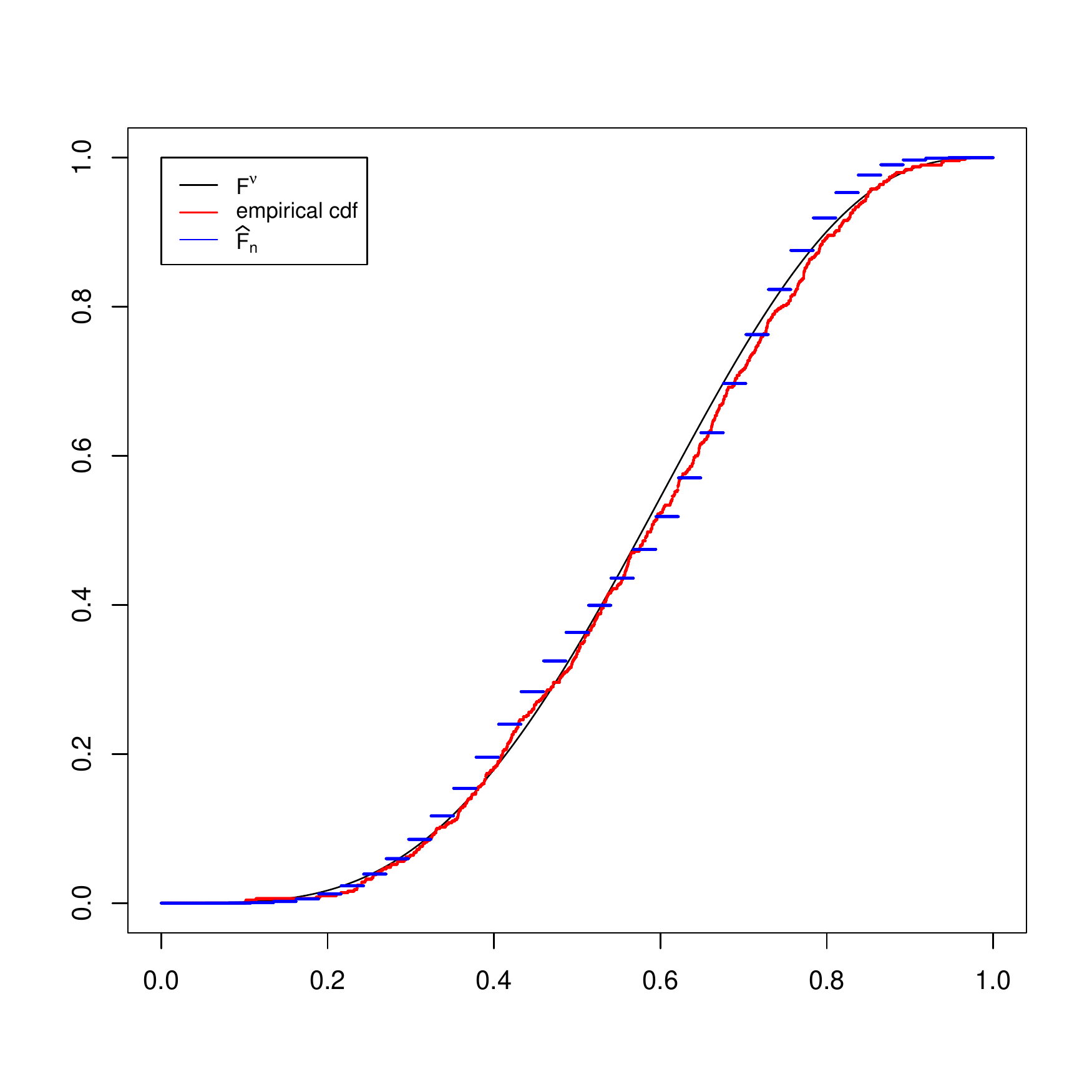}
      
      \vspace{-0.75cm}
      
      \caption{$\kappa=1$, $B(4,3)$: $n=500$, $T_n=7892$, $\est{M}=19$, $N_\infty\approx 0.105$.}\label{simu3}
      
      \vspace{0.25cm}
      
      \includegraphics[scale=0.35]{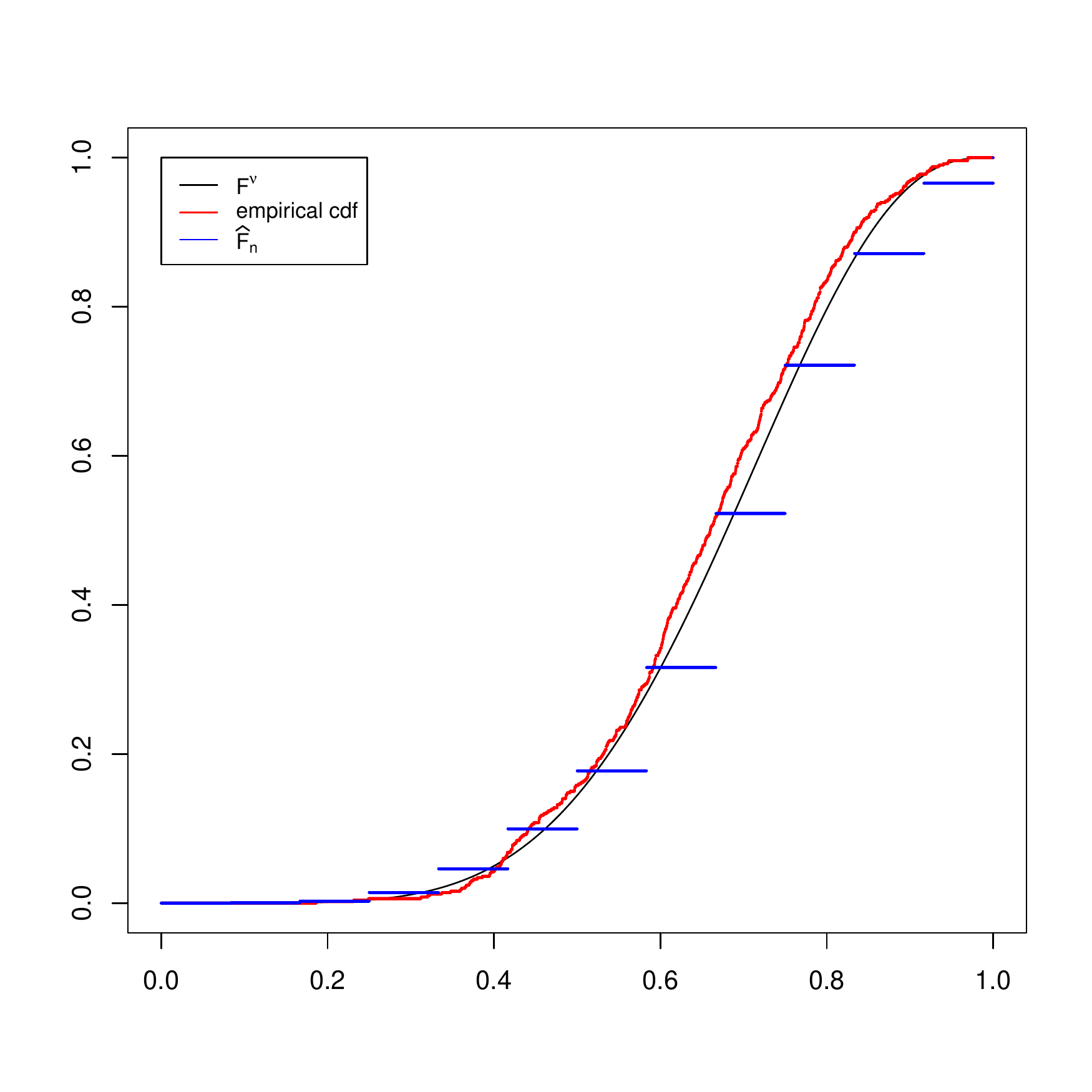}
      \caption{$\kappa=3$, $B(6,3)$: $n=500$, $T_n=1930$, $\est{M}=3$, $N_\infty\approx 0.406$.}\label{simu4}
   \end{minipage}
\end{figure}

Figures \ref{simu5} and \ref{simu6} illustrate the importance of the regularity assumption. In Figure \ref{simu5}, we consider the uniform  distribution on $[0.3,0.9]$ with $n=10000$, it shows that, at the points $0.3$ and $0.9$ where the function $F$ is non differentiable, the convergence is slower. In Figure \ref{simu6}, we consider the distribution $0.3\delta_{0.4}+0.7\delta_{0.7}$ with $n=10000$; in this case the function $F$ is not continuous and the convergence of our estimator is not clear.
\begin{figure}[h!]
	\begin{minipage}[c]{.46\linewidth}
      \includegraphics[scale=0.35]{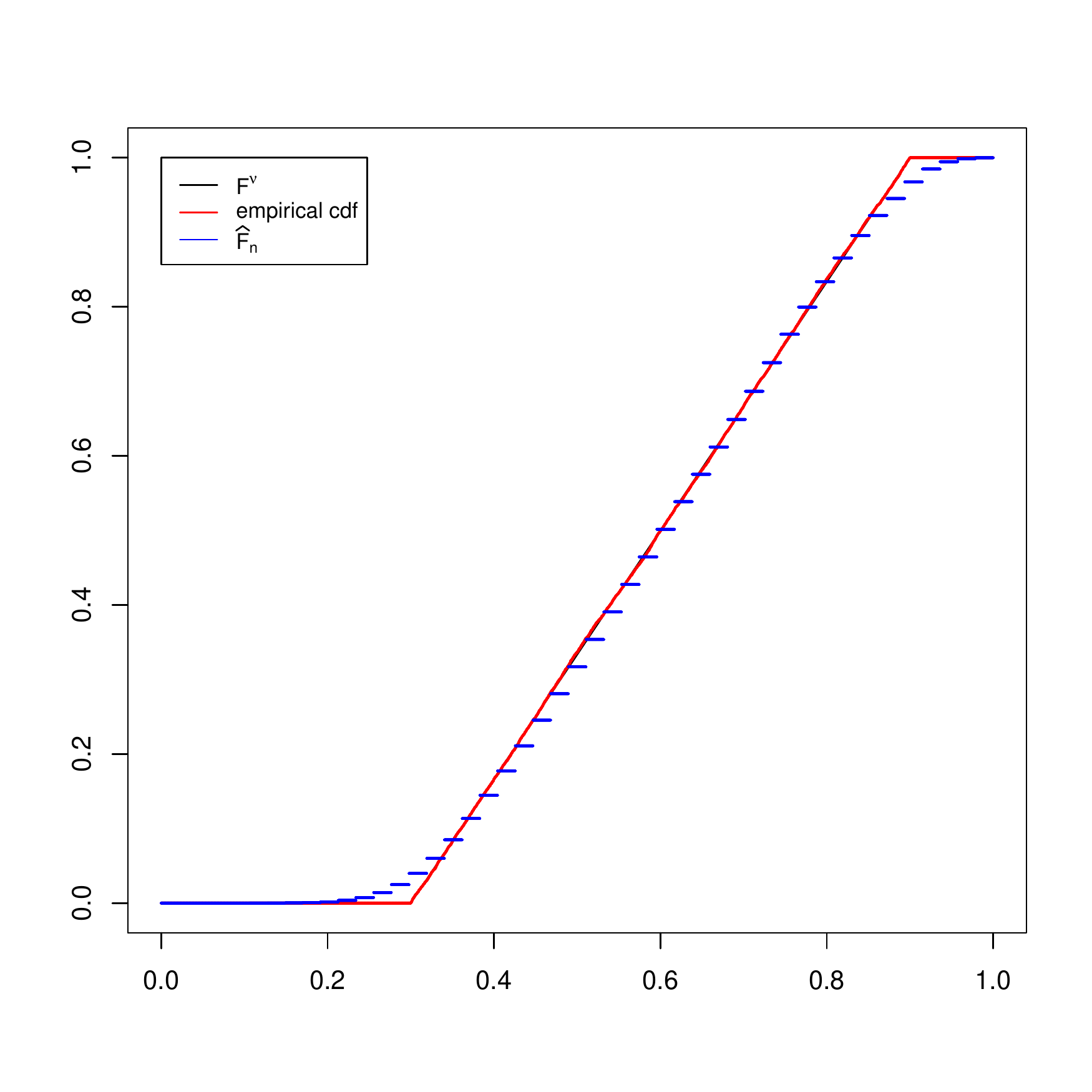}
      
      \vspace{-0.75cm}
      \caption{ $U(0.3,0.9)$: $n=10000$, $T_n=106714$, $\est{M}=46$, $N_\infty\approx0.068$.}\label{simu5}
   \end{minipage} \hfill
   \begin{minipage}[c]{.46\linewidth}
      \includegraphics[scale=0.35]{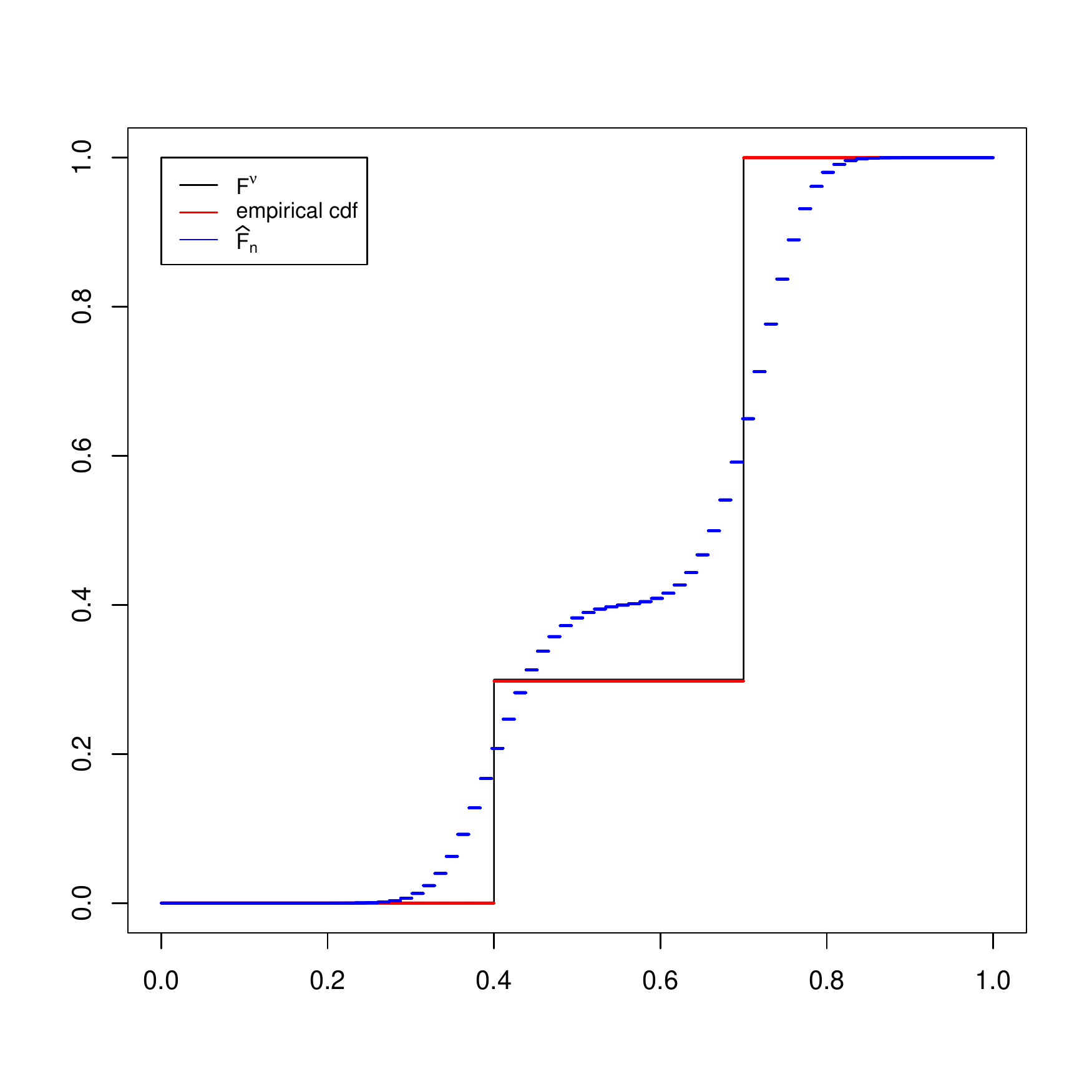}
      
      \vspace{-0.75cm}
      \caption{$D(0.4,0.7,0.3)$: $n=10000$, $T_n=66064$, $\est{M}=39$, $N_\infty\approx0.31$.}\label{simu6}
\end{minipage}
\end{figure}

Finally, in Table \ref{simu7}, for different values of $\kappa$,  the empirical mean of the loss $N_\infty(n)=\|F-\est{F}^{\est{M}}\|_\infty$ is computed on 500 simulations for any $n\in\{2^k\times100,\ k=0,\dots,7\}$ and the slope of the linear regression of $\log N_\infty(n)$ with respect to $\log n$ is compared to the theoretical bound $\paren{2+2\kappa}^{-1}$ of \thmref{cdfopt}, here we use Beta distribution $B(3+\kappa,3)$ therefore $\gamma=2$. When $\kappa=0.6$, we only compute simulations for $n\in\{2^k\times100,\ k=0,\dots,5\}$ because of computational complexity. 
Remark that the slope obtained empirically is usually better than our theoretical bound: this may follow from the use of Mc Diarmid's inequality to bound the random part of the risk of $\est{F}^M$ in \lemref{BorneRisqueF}. This bound is not optimal as seen in the control of the risk of $\est{m}^{\alpha,\beta}$ presented in \thmref{BorneRisque} and \thmref{CLTMoment}.


	
\begin{table}[h!]
      \caption{}\label{simu7}
\begin{center}
\begin{tabular}{ccc}
$\kappa$&$\paren{2+2\kappa}^{-1}$&slope\\
\hline
0.6&0.31&0.33\\
0.75&0.29&0.31\\
1&0.25&0.29\\
2&0.17&0.26\\
3&0.13&0.24
\end{tabular}
\end{center}
\end{table}

\section{Proofs}\label{Sec:proofs}
All along the proofs, $C_{\nu}$, $C_{\alpha,\nu}$ denote constants depending only on the distribution $\nu$ (or on $\nu$ and a parameter $a$), which may change from line to line.
\subsection{Proof of \thmref{CLTMoment}}\label{Sec:ProofCLTMoment}
We start by proving the convergence of 
$$
\sqrt{n}\frac{N_n^\alpha}{n}(\est{m}^{\alpha,\beta}-m^{\alpha,\beta})=\sum_{x=1}^n\frac{\Xab}{\sqrt{n}}
$$
We want to apply a central limit theorem \cite[Theorem 3.2]{HalHey80} to the martingale arrays $\left(\frac{\Xab}{\sqrt{n}}\right)_{1\leq x\leq n}$.
First, notice that, as, $\P^\nu-$a.s. $|\Xab|\leq1$,
\begin{align}\label{eq:cond1}
	\max_{1\leq i\leq n}\left|\frac{\Xab}{\sqrt{n}}\right|\cvp 0\quad\text{ and }\quad\E^\nu\cro{\max_{1\leq i\leq n}\left|\frac{1}{n}\left(\Xab\right)^2\right|}\leq1\enspace.
\end{align}
Then we only have to prove the convergence of $\frac{1}{n}\sum_{x=1}^n\left(\Xab\right)^2$ to some constant to apply the theorem. (Remark that condition (3.21) in \cite[Theorem 3.2]{HalHey80} is not necessary here as $V_{\alpha,\beta}^2$  is deterministic.) The process $\seq[x]{Z}$ is an ergodic Markov chain with invariant distribution $\pi$ defined in \eqref{eq:loiinv}. Therefore, as $\left(\Xab\right)^2$ is a bounded function of $(Z_{x-1}^n,Z_{x}^n)$ and $\paren{Z^n_x}_{0\leq x\leq n}$ has the same distribution as $\paren{Z_x}_{0\leq x\leq n}$, the mean  $\frac{1}{n}\sum_{x=1}^n\left(\Xab\right)^2$ converges $\P^\nu$-a.s. to
\begin{align*}
\sigma^2&=\E^\nu\cro{\left(\Phi_{\alpha,\beta}(\widetilde{Z}_0,\widetilde{Z}_1)-m^{\alpha,\beta}\1{\widetilde{Z}_0\geq\alpha}\right)^2}\\
&=\E^\nu\cro{\left(\Phi_{\alpha,\beta}(\widetilde{Z}_0,\widetilde{Z}_1)\right)^2}-\left(m^{\alpha,\beta}\E^\nu\cro{(1-W^{-1})^\alpha}\right)^2\enspace.
\end{align*}
Thus, according to \cite[Theorem 3.2]{HalHey80}, $\sum_{x=1}^n\frac{\Xab}{\sqrt{n}}$ converges in distribution to $\mathcal{N}(0,\sigma^2)$.
Moreover, the ergodicity of $\seq[x]{Z}$  shows also that
$N_n^\alpha/n$ converges $\P^\nu$-a.s. to $\E^\nu\cro{(1-W^{-1})^\alpha}$. \thmref{CLTMoment} follows now from Slutsky's lemma.

\subsection{Proof of \lemref{ControlEstFModel}}\label{ProofControlEstFModel}
The proof is decomposed in two parts. By the triangle inequality, 
\[\norm{\est{F}^M-F}_{\infty}\le \norm{\est{F}^M-F^M}_{\infty}+\norm{F^M-F}_{\infty}\enspace.\]
In the first part of the proof, we provide an upper bound for the random term $\norm{\est{F}^M-F^M}_{\infty}$ and in the second, an upper bound for the deterministic term $\norm{F^M-F}_{\infty}$.
\subsubsection{Control of the random part of the risk}
We use a martingale argument. For any $0\leq l\leq M+1$, let us introduce the triangular array: 
$$
\forall 1\leq x\leq n,\ Y^{M,l}_{n,x}=\psi^l_M(Z^n_{x-1},Z^n_{x})-\1{Z_x^n\geq M}F^M\paren{\frac{l}{M+1}}\enspace.
$$
\begin{lemma}\label{lem:MartingaleF} For any $n\geq0$, denote by $\seqb{\F}$ the filtration generated by the sequence $\seqb{Z}$. For any $0\leq l\leq M+1$, the triangular array  $\paren{Y^{M,l}_{n,x}}_{0\leq x\leq n}$
	is a martingale difference array with respect to $\seqb{\F}$.
\end{lemma}
\begin{proof} The transition kernel of the Markov chain $\seq[x]{Z}$ (see \eqref{eq:ker1}) yields, for any $l\leq M+1$,
\begin{align*}
	&\E^\nu[\psi^l_M(Z^n_{x},Z^n_{x+1})|\F_{x}^n]\\
	&=\1{Z_x^n\geq M}\sum_{j\geq 0}\int_0^1a^{Z^n_{x}+1}(1-a)^{j}\nu(\ud a)\frac{\binom{Z^n_{x}+j}{j}}{\binom{Z^n_{x}+j}{M}}\sum_{k=0}^{l-1}\binom{Z^n_{x}}{k}\binom{j}{M-k}\\
	&=\1{Z_x^n\geq M}\int_0^1a^{Z^n_{x}+1}\sum_{k=0}^{l-1}\sum_{j\geq M-k}\binom{Z^n_{x}}{k}\frac{\binom{Z^n_{x}+j}{j}}{\binom{Z^n_{x}+j}{M}}\binom{j}{M-k}(1-a)^{j}\nu(\ud a)\enspace.
\end{align*}
Remark that, for any $k\leq M\leq Z^n_x $ and $j\geq M-k$,
$$
\binom{Z^n_{x}}{k}\frac{\binom{Z^n_{x}+j}{j}}{\binom{Z^n_{x}+j}{M}}\binom{j}{M-k}=\binom{M}{k}\binom{Z^n_x+j-M}{Z^n_x-k}\enspace.
$$
Therefore,
\begin{align*}
	&\E^\nu[\psi^l_M(Z^n_{x},Z^n_{x+1})|\F_{x}^n]\\
	&=\1{Z_x^n\geq M}\int_0^1a^{Z^n_{x}+1}\sum_{k=0}^{l-1}\binom{M}{k}\sum_{j\geq M-k}\binom{Z^n_x+j-M}{Z^n_x-k}(1-a)^{j}\nu(\ud a)\\
	&=\1{Z_x^n\geq M}\int_0^1\sum_{k=0}^{l-1}\binom{M}{k}a^{k}(1-a)^{M-k}\nu(\ud a)=\1{Z_x^n\geq M}F^M\paren{\frac{l}{M+1}}
\end{align*}
where the second equality comes from $\eqref{eq:combfac1}$ with $\alpha=k$, $\beta=M-k$ and $i=Z^n_x$. 
\end{proof}

\lemref{BorneRisqueF} provides risk bounds for the estimators $\est{F}^M$.

\begin{lemma}\label{lem:BorneRisqueF}
	For any integers $M,n\geq 1$, and any real $z>0$,
	$$
	\P^\nu\cro{\max_{0\leq l\leq M+1}\ \absj{\est{F}^M\paren{\frac{l}{M+1}}-F^M\paren{\frac{l}{M+1}}}\geq \frac{n}{N_{n}^{M}}\sqrt{\frac{z+\log M}{2n}}}\leq 2e^{-z}\enspace.
	$$
\end{lemma}
\begin{proof}
First, remark that 
\[
\max_{0\leq l\leq M+1}\ \absj{\est{F}^M\paren{\frac{l}{M+1}}-F^M\paren{\frac{l}{M+1}}}=\frac{1}{N^M_n}\max_{1\leq l\leq M}\ \absj{\sum_{x=1}^nY^{M,l}_{n,x}}\enspace .
\]
Moreover, for any $1\leq l\leq M$ and $i,j\geq0$,
\begin{align*}
	\sum_{k=0}^{l-1}\binom{i}{k}\binom{j}{M-k}&\leq \sum_{k=0}^{M}\binom{i}{k}\binom{j}{M-k}=\binom{i+j}{M}
\end{align*}
then $\psi^l_M\in[0,1]$ and $\absj{Y^{M,l}_{n,x}}\leq 1$.
 Thus, Mc Diarmid's inequality (Theorem 6.7 in \cite{McD89}) and \lemref{MartingaleF} yield for any $1\leq l\leq M$,
  \begin{align*}
 	\P^\nu\cro{\frac{1}{N^M_n}\absj{\sum_{x=1}^nY^{M,l}_{n,x}}\geq \frac{n}{N_{n}^{M}}\sqrt{\frac{z}{2n}}}\leq 2e^{-z}\enspace .
 \end{align*}
The result of the lemma follows now from a union bound.
\end{proof}

\subsubsection{Control of the bias}
Let us now turn to the term $\|F-F^M\|_\infty$. The rate of convergence depends on the H\"older regularity of $F$. 

\begin{lemma}\label{lem:biasF} Suppose that the function $F$ is in $\mathcal{C}^\gamma$ for some $\gamma\in(0,2]$. For any integer $M\geq 0$,
	$$
	\max_{0\leq l\leq M+1}\ \absj{F\paren{\frac{l}{M+1}}-F^M\paren{\frac{l}{M+1}}}\leq \frac{\|F\|_{\gamma}}{2^\gamma(M+2)^{\gamma/2}}\enspace.
	$$
\end{lemma}
%
\begin{proof} We adapt the proof of Theorem 2 in \cite{Mnatsakanov:2008}.
An integration by parts shows that, for any $l\in\set{1,\dots,M}$,
\begin{align}\label{eq:Fbeta}
F^M\paren{\frac{l}{M+1}}=\int_0^1F(u)b_{l,M+1-l}(u)\ud u
\end{align}
where 
$$b_{l,M+1-l}(u)=\frac{M!}{(l-1)!(M-l)!}u^{l-1}(1-u)^{M-l}$$ is the probability density function of the beta-distribution with parameters $l$ and $M+1-l$. We then introduce a random variable $B_{l,M+1-l}$, with density $b_{l,M+1-l}$. Recall that expectation and variance of $B_{l,M+1-l}$ are given respectively by
$$
\E\cro{B_{l,M+1-l}}=\frac{l}{M+1}\quad \text{and} \quad\var{B_{l,M+1-l}}=\frac{l(M+1-l)}{(M+1)^2(M+2)}\enspace.
$$ 
Suppose that $\gamma>1$. According to \eqref{eq:Fbeta},
\begin{align*}
	&F^M\paren{\frac{l}{M+1}}-F\paren{\frac{l}{M+1}}=\\
	&\E\cro{F(B_{l,M+1-l})-F\paren{\frac{l}{M+1}}-F'\paren{\frac{l}{M+1}}\paren{B_{l,M+1-l}-\frac{l}{M+1}}}\enspace.
\end{align*}
As the function $F$ is $\gamma$-H\"older,
\begin{align*}
	\absj{F^M\paren{\frac{l}{M+1}}-F\paren{\frac{l}{M+1}}}\leq&\|F\|_{\gamma}\E\cro{\absj{B_{l,M+1-l}-\frac{l}{M+1}}^\gamma}
\end{align*}
and H\"older's inequality leads to
\begin{align*}
	\absj{F^M\paren{\frac{l}{M+1}}-F\paren{\frac{l}{M+1}}}\leq&\|F\|_{\gamma}\var{B_{l,M+1-l}}^{\gamma/2}\leq\frac{\|F\|_{\gamma}}{2^\gamma(M+2)^{\gamma/2}}\enspace.
\end{align*}
The argument is easily adapted for the case $\gamma\leq1$.
\end{proof}
\subsubsection{Conclusion of the proof of \lemref{ControlEstFModel}}
	As $\|\est{F}^M-F\|_\infty$ can be written as
	\begin{align*}
		\|\est{F}^M-F\|_\infty&=\max_{0\leq l\leq M}\sup_{u\in\left[\frac{l}{M+1},\frac{l+1}{M+1}\right[} \absj{\est{F}^M\paren{\frac{l}{M+1}}-F\paren{u}}\\
		\leq& \max_{0\leq l\leq M} \absj{\est{F}^M\paren{\frac{l}{M+1}}-F^M\paren{\frac{l}{M+1}}}\\
		&+\max_{0\leq l\leq M} \absj{F^M\paren{\frac{l}{M+1}}-F\paren{\frac{l}{M+1}}}\\
		&+\max_{0\leq l\leq M}\sup_{u\in\left[\frac{l}{M+1},\frac{l+1}{M+1}\right[}\absj{F(u)-F\paren{\frac{l}{M+1}}}\enspace,
	\end{align*}
	the result follows immediately from \lemref{BorneRisqueF}, \lemref{biasF} and the fact that $F$ is $\gamma\wedge1$-H\"older.

\subsection{Construction of $\est{F}^{\est{M}^z}$ and proof of \lemref{Lepski}}\label{ConstructionofestF}
We follow the construction of Lepskii \cite{Lepski:1991}. A union bound in \lemref{BorneRisqueF} shows that, for $C=\pi^2/3$,
\begin{equation}\label{eq:ContralVarUnionBound}
\P^\nu\cro{\forall M\ge 1,\ \norm{\est{F}^M-F^M}_{\infty}\leq \frac{n}{N_{n}^{M}}\sqrt{\frac{z+3\log M}{2n}}}\ge 1-Ce^{-z}\enspace. 
\end{equation}
Moreover, by \lemref{biasF} and the fact that $F$ is $\gamma\wedge1$-H\"older
\begin{equation}\label{eq:ControlBiasF}
\norm{F-F^M}_{\infty}\le \frac{2\|F\|_\gamma}{(M+1)^{\gamma/2}}\enspace. 
\end{equation}
Fix some real $z>0$ and define for any integer $M\geq 1$,
\[\Delta(M)=\sup_{M'\ge 1}\set{\norm{\est{F}^{M'}-\est{F}^{M\wedge M'}}-\frac{2n}{N_{n}^{M'}}\sqrt{\frac{z+3\log M'}{2n}}}\enspace.\]
The random variable $\est{M}^z$ is defined by
\[\est{M}^z=\arg\min_{M\ge 1}\set{\Delta(M)+\frac{2n}{N_{n}^{M}}\sqrt{\frac{z+3\log M}{2n}}}\enspace.\]


We now have to check that $\est{M}^z$ satisfies the inequality of \lemref{Lepski}. Let 
\[\Omega=\set{\forall M\ge 1,\ \norm{\est{F}^M-F^M}_{\infty}\leq \frac{n}{N_{n}^{M}}\sqrt{\frac{z+3\log M}{2n}}}\enspace.\]
By \eqref{eq:ContralVarUnionBound}, $\P^\nu\cro{\Omega} \ge 1-Ce^{-z}$. Denote 
\[\forall M\ge 1,\qquad \est{R}(M)=\frac{n}{N_{n}^{M}}\sqrt{\frac{z+3\log M}{2n}}\enspace.\]
On $\Omega$, by the triangle inequality,
\begin{align*}
 \norm{\est{F}^{\est{M}^z}-F}_{\infty}&\le \norm{\est{F}^{\est{M}^z}-\est{F}^{\est{M}^z\wedge M}}_{\infty}+\norm{\est{F}^{M}-\est{F}^{\est{M}^z\wedge M}}_{\infty}+\norm{\est{F}^{M}-F}_{\infty}\\
 &\le \Delta(M)+2\est{R}(\est{M}^z)+\Delta(\est{M}^z)+2\est{R}(M)+\frac{2\|F\|_\gamma}{(M+1)^{\gamma/2}}\\
 &\le 2(\Delta(M)+2\est{R}(M))+\frac{2\|F\|_\gamma}{(M+1)^{\gamma/2}}\enspace.
\end{align*}
Now, using the triangle inequality once again, for any $M'\ge M$,
\begin{align*}
\norm{\est{F}^{M'}-\est{F}^{M}}_{\infty}-2\est{R}(M')&\le\paren{\norm{\est{F}^{M'}-F_{M'}}-\est{R}(M')}\\
&+\norm{F_{M'}-F_{M}}+\paren{\norm{\est{F}^{M}-F^M}-\est{R}(M')}\enspace.
\end{align*}
The first term is non positive on $\Omega$, the second one is bounded by $\frac{4\|F\|_\gamma}{(M+1)^{\gamma/2}}$ by \eqref{eq:ControlBiasF} and the third one is non positive on $\Omega$ since $\est{R}$ is non decreasing. It follows that $\Delta(M)\le \frac{4\|F\|_\gamma}{(M+1)^{\gamma/2}}$ and the proof is complete.

\subsection{Asymptotic of $N^M_n/n$}\label{Sec:AsNnM}
 We start with the transient case.
\begin{lemma}\label{lem:AsyNnM}Suppose that $\nu$ satisfies Assumption \eqref{hypderdroite} and $\E^\nu\cro{\log\rho_0}<0$. There is a constant $C_\nu$ such that, for any integers $M\geq0$, $n\ge 1$ and any real $z>0$, 
	\begin{align*}
 &\P^\nu\paren{\absj{\frac{N^M_n}{n}-\pi\paren{[M,\infty)}}\geq C_\nu\sqrt{\frac{z}{n}}}\leq 2 e^{-z}\enspace.
	\end{align*}
\end{lemma}
\begin{proof}
We will apply the concentration inequality for Markov chains \cite[Theorem 0.2]{DedGou15} to $\sum_{x=0}^{n-1}\1{Z_x\geq M}$. As $\seq[x]{Z}$ is an irreducible aperiodic Markov chain on a countable state space, we only have to prove that it is geometrically ergodic. To this purpose, we prove that the return time to 0 
$T=\inf\set{x\geq1,\ Z_x=0}$ has an exponential moment. By Lemma 2 in \cite{KesKozSpi:1975}, there is a constant $C_\nu$ such that for any $t\geq0$,
  $$
  \P^\nu\paren{T>t}\leq C_\nu e^{-\frac{t}{C_\nu}}\enspace.
  $$
Therefore,
\[\E^\nu\cro{e^{\frac{T}{2C_{\nu}}}}=\int_{0}^{+\infty}\frac1{2C_\nu}e^{\frac t{2C_\nu}}\P^\nu\paren{T>t}\ud t<\infty\enspace.\]
Hence, the Markov chain $(Z_x)_{x\in\Z_+}$ is geometrically ergodic and by \cite[Theorem 0.2]{DedGou15}, there exists a constant $C_\nu$ such that for any real $x>0$ and any integer $M\geq0$,
$$
\P^\nu\paren{\absj{\frac{N^M_n}{n}-\pi\paren{[M,\infty)}}\geq x}\leq 2 e^{-\frac{nx^2}{C_\nu}}
	$$
	
The result of the lemma follows
\end{proof}

The behavior  of the tails of $\pi$ is given by the following lemma.

\begin{lemma}\label{lem:TailPi}Suppose that $\nu$ satisfies Assumption \eqref{hypderdroite} and $\E^\nu\cro{\log\rho_0}<0$. There is a constant $C_\nu$ such that when $M$ tends to $\infty$,
	$$\pi\paren{[M,\infty)}\sim \frac{C_\nu}{M^\kappa}\enspace.$$
Therefore, up to a change of the constant $C_\nu$, for any $M\ge 1$, 
\[\pi\paren{[M,\infty)}\ge \frac{C_\nu}{M^\kappa}\enspace.\]
\end{lemma}
\begin{proof}
	By \cite{KesKozSpi:1975} Lemma 1,
	\begin{equation}\label{eq:queueS}
		\P^\nu\paren{W\geq x}\sim \frac{C_\nu}{x^\kappa}\enspace.
	\end{equation}

	The definition of $\pi$ given in \eqref{eq:loiinv}, 
	\begin{align*}
	\pi\paren{[M,\infty)}=&\E^\nu\cro{(1-W^{-1})^{M}}\\
	=&\int_0^1\P^\nu\paren{(1-W^{-1})^M\geq u}\ud u\\
	=&\int_0^1\P^\nu\paren{W\geq \frac{1}{1-u^{\frac{1}{M}}}}\ud u
	\end{align*}
According to \eqref{eq:queueS}, for any $u>0$, using that $u^{\frac 1M}=e^{\frac1M\log u}\ge1+\frac1M\log u $, we get
\begin{align*}
	M^\kappa\P^\nu\paren{W\geq \frac{1}{1-u^{\frac{1}{M}}}}&\leq C_\nu \paren{M-Mu^{\frac{1}{M}}}^{\kappa}\leq C_\nu(\log 1/u)^\kappa\enspace.
\end{align*}
As $\lim_{M\to\infty}M^\kappa\P^\nu\paren{W\geq \frac{1}{1-u^{\frac{1}{M}}}}=C_\nu(\log 1/u)^\kappa$, dominated convergence theorem gives the result.
\end{proof}

%

To deal with the recurrent regime, as there is no invariant probability in this case, we use the following lemma.

\begin{lemma}\label{lem:AsyNnMrec}
	Suppose that $\nu$ satisfies Assumption \eqref{hypderdroite} and that $\E^\nu\cro{\log\rho_0}=0$. Then, for any $a>0$, there is a constant $C_{a, \nu}$ such that for any integer $n\geq2$,
$$\P^\nu\paren{N_n^{n^a}< \frac{n}2}\leq C_{a,\nu}\frac{\log n}{\sqrt{n}}\enspace .$$
\end{lemma}

\begin{proof} 
Let $V_{x}=\sum_{i=0}^{x}\log\rho_{i}$ and $W_x=\sum_{y=0}^{x}e^{V_{x}-V_{y}}$. Given $\omega$, $Z_x+1$ follows the geometric distribution $\mathcal{G}\paren{W_x^{-1}}$ (see the proof of Theorem 4.5 in \cite{CoFaLoLoMa2014}). Hence, 
	\begin{align}\label{eq:majZn}
		\P_\omega\paren{Z_x<n^a}= 1-(1-1/W_x)^{n^a}\leq \frac{n^a }{W_x}\leq n^a e^{-\paren{V_{x}-\min_{y\leq x}V_{y}}}\enspace.
	\end{align}
We start by proving that, with large probability, $\paren{V_{x}-\min_{y\leq x}V_{y}}$ is larger than $(2+a) \log n$ for many sites $x$. More precisely, consider the event
$$E_n=\set{\sum_{x=1}^n\1{V_{x}-\min_{y\leq x}V_{y}\leq (2+a)\log n}\geq \frac{n}2}\enspace.$$
Markov inequality yields
\begin{align}\label{eq:Markineq}
	\P^\nu\paren{E_n}\leq& \frac{2}{n}\sum_{x=1}^n\P^\nu\paren{V_{x}-\min_{y\leq x}V_{y}\leq (2+a)\log n}\enspace.
\end{align}

As the variables $\rho_i$ are i.i.d., for a fixed value $x$, $\paren{V_{x}-V_{y}}_{0\leq y\leq x}$ has the same distribution as $\paren{V_{x-y-1}}_{0\leq y\leq x}$ then,
\begin{align}\label{eq:lawlog}
	\P^\nu\paren{V_{x}-\min_{y\leq x}V_{y}\leq (2+a)\log n}=\P^\nu\paren{\max_{0\leq y\leq x}V_{y}\leq (2+a)\log n}
\end{align}
We now have to control the random variables $\max_{0\leq y\leq x} V_{y}$, $1\leq x\leq n$. For this purpose, we use the Koml\'os-Major-Tusn\'ady strong approximation theorem (see \cite[Theorem 1]{KMT:1976}): denote by $\sigma^2$ the variance of $\log \rho_0$, on a possibly enlarged probability space, there exists a Brownian motion $B$ and two constants $c_{1,\nu}$ and $c_{2,\nu}$ independent of $n$ such that
	$$
	\P^\nu\paren{\max_{y\in\cro{0,n}}|V_{\lfloor y\rfloor}-\sigma B_y|\geq c_{1,\nu} \log n}\leq \frac{c_{2,\nu}}{n}\enspace.
	$$
	Therefore,
	\begin{multline}\label{eq:KMT}
		\sum_{x=1}^n\P^\nu\paren{\max_{0\leq y\leq x}V_{y}\leq (2+a)\log n}\\
		\leq c_{2,\nu}+\sum_{x=1}^n\P^\nu\paren{\max_{y\in[0,x]}B_{y}\leq \frac{(2+a)+c_{1,\nu}}{\sigma}\log n}\enspace . 
	\end{multline}
By the reflection principle \cite[Proposition 3.7 in Ch III]{Rev_Yor:1999}, $\max_{ y\in\cro{0,x}} B_{y}$ has the same distribution as $|B_{x}|$. Thus, there exists a constant $C_{a,\nu}$ depending only on $\nu$ and $a$ such that
\begin{align}\label{eq:reflexion}
	\P^\nu\paren{\max_{y\in[0,x]}B_{y}\leq \frac{(2+a)+c_{1,\nu}}{\sigma}\log n}&=\P^\nu\paren{|B_{1}|\leq \frac{((2+a)+c_{1,\nu})\log n}{\sigma\sqrt{x}}}\notag\\
	&\leq C_{a,\nu}\frac{\log n}{\sqrt{x}}
\end{align}

Equations \eqref{eq:Markineq}, \eqref{eq:lawlog},\eqref{eq:KMT} and \eqref{eq:reflexion} lead to
\begin{align}\label{eq:majenv}
	\P^\nu\paren{E_n}\leq C_{a,\nu}\frac{\log n}{\sqrt{n}}\enspace.
\end{align}

On the complementary event $\overline{E}_n$, the set 
$$I_{n}=\set{x\in\set{1,\dots,n},\ V_{x}-\min_{y\leq x}V_{y}> (2+a)\log n}$$
has at least $n/2$ elements. Moreover, according to \eqref{eq:majZn}, for any $x\in I_n$,
\begin{align*}
		\P_\omega\paren{\exists x\in I_n, Z_x<n^a}\leq \sum_{x\in I_n} \P_\omega\paren{Z_x<n^a}\leq \frac1n\enspace.
	\end{align*}
Therefore, on $\overline{E}_n$,
\begin{align}\label{eq:majZn2}
		\P_\omega\paren{\sum_{x=0}^{n-1}\1{Z^n_x< n^a}> \frac{n}2}\leq \frac1n\enspace.
	\end{align}

It is now easy to conclude the proof of the lemma. Indeed,
\begin{align*}
	\P^\nu\paren{N_n^{n^a}< \frac{n}2}&=\P^\nu\paren{\sum_{x=0}^{n-1}\1{Z^n_x\geq n^a}< \frac{n}2}=\P^\nu\paren{\sum_{x=0}^{n-1}\1{Z^n_x< n^a}> \frac{n}2}\\
	&\leq \P^\nu\paren{E_n}+\P^\nu\paren{\overline{E}_n\cap\set{\sum_{x=0}^{n-1}\1{Z^n_x< n^a}> \frac{n}2}}\enspace.
\end{align*}
Equations \eqref{eq:majenv} and \eqref{eq:majZn2} give now the result.
\end{proof}

\subsection{Conclusion of the proof of \thmref{cdfopt}} \label{ConclusionProof}

\subsubsection*{Transient case}

By \lemref{AsyNnM} and a union bound, $\P^\nu\paren{\Omega_1}\ge 1-C_\nu e^{-z}$, where
\[\Omega_1=\set{\forall M\ge 1,\qquad \absj{\frac{N^M_n}{n}-\pi\paren{[M,\infty)}}\leq C_\nu\sqrt{\frac{z+\log M}{n}}}\enspace.\]
On $\Omega_1$, for any $M$ such that $\pi([M,+\infty))\ge 2C_\nu\sqrt{\frac{z+\log M}{n}}$
\[\frac n{N_n^M}\le \frac{C_\nu}{\pi([M,+\infty))}\enspace.\]
Therefore, by the second part of \lemref{TailPi}, on $\Omega_1$, for any $M$ such that $M^{-\kappa}\ge C_\nu\sqrt{\frac{z+\log M}{n}}$
\[\frac n{N_n^M}\ge C_\nu M^\kappa\enspace.\]
By \lemref{Lepski}, $\P^\nu\paren{\Omega_2}\ge 1-(\pi^2/3)e^{-z}$, where,
\[\Omega_2=\set{ \norm{\est{F}^{\est{M}^z}-F}_{\infty}\le \inf_{M\ge 1}\set{ \frac{6\|F\|_\gamma}{(M+1)^{\gamma/2}}+\frac{4n}{N_{n}^{M}}\sqrt{\frac{z+3\log M}{2n}}}} \enspace.\]
Therefore, on $\Omega=\Omega_1\cap\Omega_2$,
\begin{align*}
\norm{\est{F}^{\est{M}^z}-F}_{\infty}&\le C_{\nu}\inf_{1\le M\le C_{\nu}(\frac{n}{z+\log M})^{1/(2\kappa)}}\set{ \frac{1}{M^{\gamma/2}}+M^\kappa\sqrt{\frac{z+\log M}{n}}}\\
&\le C_\nu\paren{\frac{z+\log n}{n}}^{\frac{\gamma}{2\gamma+4\kappa}}\enspace. 
\end{align*}
For the result in expectation, we only have to take $z=\log n$.

\subsubsection*{Recurrent case}
By \lemref{Lepski}, taking $z=\log n$ and $M=n^{1/\gamma}$, $\P^\nu\paren{\Omega_1}\ge 1-\frac{\pi^2}{3n}$ where,
\[\Omega_1=\set{ \norm{\est{F}^{\est{M}^{\log n}}-F}_{\infty}\le \frac{6\|F\|_\gamma}{\sqrt{n}}+\frac{4n}{N_{n}^{n^{1/\gamma}}}\sqrt{\frac{(1+3/\gamma)\log n}{2n}}} \enspace.\]

By \lemref{AsyNnMrec}, $\P^\nu\paren{\Omega_2}\ge 1-C_{\gamma,\nu}\frac{ \log n}{\sqrt{n}}$, where $\Omega_2=\set{N_{n}^{n^{1/\gamma}}\geq n/2}$.
Therefore,
\begin{align*}
	\E^\nu\cro{\norm{\est{F}^{\est{M}^{\log n}}-F}_{\infty}}&\leq\E^\nu\cro{\norm{\est{F}^{\est{M}^{\log n}}-F}_{\infty}\mathbf{1}_{\Omega_1\cap \Omega_2}}+\P^\nu\paren{\overline{\Omega}_1}+\P^\nu\paren{\overline{\Omega}_2}\\
	&\leq C_{\gamma,\nu} \frac{\log n}{\sqrt{n}}\enspace.
\end{align*}

\bibliographystyle{alpha}
\bibliography{BT}

\end{document}